\documentclass[12pt,regno]{amsart}
\usepackage{epsf,epsfig,amsmath}
\usepackage{amssymb,latexsym}

\usepackage{amsthm} 
\usepackage{amsfonts}
\usepackage{graphicx,psfrag}

\numberwithin{equation}{section}

\newtheorem{thm}{Theorem}[section]
\newtheorem{pro}[thm]{Proposition}
\newtheorem{lem}[thm]{Lemma}
\newtheorem{con}[thm]{Conjecture}
\newtheorem{cor}[thm]{Corollary}

\newtheorem{rem}[thm]{Remark}
\newtheorem{defi}[thm]{Definition}

\newtheorem*{cor*}{Corollary}
\newtheorem*{thm*}{Theorem}

\title[Special matchings and parabolic K--L polynomials]{Special matchings and 
parabolic Kazhdan--Lusztig polynomials}

\author{Mario Marietti}

\address{Dipartimento  di Ingegneria Industriale e Scienze Matematiche, Universit\`a Politecnica delle Marche, Via Brecce Bianche, 60131 Ancona,  Italy}

\email{m.marietti@univpm.it}


\begin{document}

\maketitle

\begin{abstract}
We prove that the combinatorial concept of a special matching can be used to compute the parabolic Kazhdan--Lusztig polynomials of doubly laced Coxeter groups and of dihedral Coxeter groups. In particular, for this class of groups which  includes all Weyl groups, our results generalize to the parabolic setting the main results in [Advances in Math. {202} (2006), 555-601]. As a consequence, the parabolic Kazhdan--Lusztig polynomial indexed by $u$ and $v$  depends only on the poset structure of the Bruhat interval from the identity element to $v$ and on which elements of that interval are minimal coset representatives.
\end{abstract}

\section{Introduction}
Kazhdan--Lusztig polynomials are polynomials $\{P_{u,v}(q)\}_{u,v\in W}$ in one variable $q$, which are indexed by a pair of elements $u,v$ in a Coxeter group $W$. These polynomials were introduced by Kazhdan and Lusztig in \cite{K-L} as a tool for the construction of certain important representations of the Hecke algebra associated with $W$. Since then, Kazhdan--Lusztig polynomials have been shown to have applications in many contexts and now play a central role in Lie theory and representation theory.
In particular, when $W$ is a Weyl group, the Kazhdan--Lusztig polynomial $P_{u,v}(q)$ is the Poincar\'e polynomial of the local intersection cohomology groups (in even degrees) of the Schubert variety associated with $v$ at any point of the Schubert variety associated with $u$ (see \cite{KL2}).
Thus, for Weyl groups, the coefficients of  $P_{u,v}(q)$ are nonnegative, a fact which is not at all evident from the definition of $P_{u,v}(q)$. 
In fact, Elias and  Williamson \cite{E-W} have recently shown, much more generally, that the coefficients of  Kazhdan--Lusztig polynomials are nonnegative for every Coxeter group.

At present, from a combinatorial point of view, the most challenging conjecture about Kazhdan--Lusztig polynomials  is arguably the following, made by Lusztig in private and, independently, by Dyer \cite{Dyeth}.
\begin{con}
\label{comb-inv-con}
The Kazhdan--Lusztig polynomial $P_{u,v}(q)$ depends only on the combinatorial structure of the Bruhat interval $[u,v]$ (i.e., the isomorphism type of $[u,v]$ as a poset under Bruhat order). 
\end{con}
Conjecture~\ref{comb-inv-con} is usually referred to as the Combinatorial Invariance Conjecture. It is equivalent to the analogous conjecture on the Kazhdan--Lusztig $R$-polynomials. These are also polynomials $\{R_{u,v}(q)\}_{u,v\in W}$ in one variable $q$,  indexed by a pair of elements $u,v$ in a Coxeter group $W$, and were introduced by Kazhdan--Lusztig in the same work \cite{K-L}. The knowledge of the entire family $\{R_{u,v}(q)\}_{u,v\in W}$ of the Kazhdan--Lusztig $R$-polynomials of a Coxeter group $W$ is equivalent to the  knowledge of the entire family $\{P_{u,v}(q)\}_{u,v\in W}$ of the Kazhdan--Lusztig polynomials of $W$. 

 The Combinatorial Invariance Conjecture asserts that, given two Coxeter groups $W$ and $W'$ and two pairs of elements $u,v\in W$ and $u',v' \in W'$ such that $[u,v]\cong[u',v']$, it holds that $R_{u,v}(q) = R_{u',v'}(q)$ and  $P_{u,v}(q) = P_{u',v'}(q)$. This was proved  in \cite{BCM1} to hold when $u$ and $u'$ are the identity elements of $W$ and $W'$. The proof of this result is constructive since it describes an algorithm to compute the Kazhdan--Lusztig $R$-polynomial $R_{u,v}(q)$ depending only on the lower Bruhat interval $[e,v]$ (where $e$ denotes the identity element). This algorithm is based on combinatorial tools named {\em special matchings}, which are abstractions of the maps 
given by the multiplication (on the left or on the right) by a Coxeter generator. 
A special matching of $[e,v]$ is an involution
\( M:[e,v]\rightarrow [e,v] \) such that 
\begin{enumerate}
\item either \( u \lhd M(u)\) or  \( u \rhd M(u)\), for all \( u\in [e,v]  \),
\item  if $u_1\lhd u_2$ then  $M(u_1)\leq M(u_2),$
 for all \( u_1,u_2\in [e,v] \) such that \( M(u_1)\neq u_2 \).
\end{enumerate}

(Here, $\lhd$ denotes the covering relation, i.e., $x\lhd y$ means that $x<y$ and there is no $z$ with $x<z<y$). The concept of special matching is purely poset-theoretic, that is, it depends only on the poset structure of $[e,v]$ and not on other structures (in particular,  the algebraic structure of group plays no role).
In \cite{BCM1}, it is proved that special matchings may be used in place of 
multiplication maps in the recurrence formula which computes the $R$-polynomials.
 
In order to find a method for the computation of the dimensions of the intersection cohomology
modules corresponding to Schubert varieties in $G/P$, where $P$ is a parabolic subgroup of the
Kac--Moody group $G$, Deodhar \cite{Deo87} defined two parabolic analogues of the Kazhdan--Lusztig polynomials,
which correspond to the roots $x=q$ and $x=-1$ of the equation $x^2=q+ (q-1)x$. Also, Deodhar  defined two parabolic analogues of the $R$-polynomials, denoted $\{R_{u,w}^{H,x} (q)\}_{u,w\in W^H}$, whose knowledge is again equivalent to the knowledge of the parabolic Kazhdan--Lusztig polynomials. These polynomials are indexed by pairs of elements in the set $W^H$ of minimal coset representatives with respect to the standard parabolic subgroup $W_H$ generated by a subset $H\subseteq S$. The parabolic Kazhdan--Lusztig and $R$-polynomials  coincide with  the ordinary Kazhdan--Lusztig and $R$-polynomials when $H= \emptyset$.

Since the appearance of \cite{BCM1}, the authors have been asked many times  whether something analogous could be done in the parabolic setting, i.e., whether  a poset-theoretic way to compute the parabolic Kazhdan--Lusztig polynomials could be found.
In this work, we generalize the main results in \cite{BCM1} to the parabolic setting, when $(W,S)$ is a doubly laced Coxeter system (and, also, in the case of dihedral Coxeter systems, i.e. Coxeter systems of rank 2, which is much easier). 

Let $(W,S)$ be a Coxeter system which is either doubly laced or dihedral. Let $H \subseteq S$ and $w\in W^H$. Let us consider the special matchings $M$ of the lower Bruhat interval $[e,w]$ satisfying the following further property:
$$u \leq w, u \in W^H, M(u) \lhd u \Longrightarrow M(u) \in W^H.$$
We call such matchings {\em $H$-special}. Note that the concept of $H$-special matching depends both on the poset structure of the complete interval $[e,w]$ and on how the parabolic interval $[e,w]^{H} = \{ z \in W^{H}: \;  e \leq z \leq w \}$ embeds in $[e,w]$. 
In this work, we show that the $H$-special matchings  may be used in place of left
multiplication maps in the recurrence formula for the parabolic Kazhdan--Lusztig $R$-polynomials: precisely, if $M$ is an $H$-special matching of $[e,w]$,  we have
\begin{equation}
 \label{calcola}
 R_{u,w}^{H,x} (q)= \left\{ \begin{array}{ll}
R_{M(u),M(w)}^{H,x}(q), & \mbox{if $M(u)  \lhd u$,} \\
(q-1)R_{u,M(w)}^{H,x}(q)+qR_{M(u),M(w)}^{H,x}(q), & \mbox{if $M(u) \rhd u$
and $M(u) \in W^{H}$,} \\
(q-1-x)R_{u,M(w)}^{H,x}(q), & \mbox{if $M(u) \rhd u$ and  $M(u) \notin W^{H}$.} 
\end{array} \right. 
\end{equation}
The parabolic $R$-polynomials can be computed by means of $H$-special matchings by iterating (\ref{calcola}).
As a corollary, the parabolic $R$-polynomial $ R_{u,w}^{H,x} (q)$ depends only on the poset structure of the complete interval $[e,w]$ and on which of the elements in this interval are minimal coset representatives.
Indeed, we have the following result (Corollary~\ref{congettura} in the paper).
\begin{thm*}
Let $(W_1,S_1)$ and $(W_2,S_2)$ be two doubly laced or dihedral Coxeter systems, with identity elements $e_1$ and $e_2$, and let $H_1 \subseteq S_1$ and 
$H_2\subseteq S_2$. Let $v_1 \in W_1^{H_1}$ and $v_2 \in W_2^{H_2}$ be such that there exists 
a poset-isomorphism  $\psi$ from $[e_1,v_1]$ to $[e_2,v_2]$ which restricts to a 
poset-isomorphism from $[e_1,v_1]^{H_1}$ to $[e_2,v_2]^{H_2}$. Then, for all $u,w \in [e_1,v_1]^{H_1}$, we have 
$$R_{u,w}^{H_1,x} (q)=R_{\psi(u),\psi(w)}^{H_2,x} (q)  \quad \text{ and } \quad  P_{u,w}^{H_1,x} (q)=P_{\psi(u),\psi(w)}^{H_2,x} (q).$$  
\end{thm*}
Since the $\emptyset$-special matchings are exactly the special matchings, the preceding result implies the main result of \cite{BCM1} for the ordinary Kazhdan--Lusztig and $R$-polynomials.

In the proofs, we use  some algebraic properties of the special matchings of a  lower Bruhat interval $[e,w]$ which are valid for any arbitrary Coxeter group $W$ (see \cite{Mchara}), while the further hypotheses on $W$ are needed only in few cases.
We believe that the main result of this work might be generalized.

\par

In studying the parabolic Kazhdan--Lusztig and $R$-polynomials $\{P_{u,w}^{H,x} (q)\}_{u,w\in W^H}$ and  $\{R_{u,w}^{H,x} (q)\}_{u,w\in W^H}$, attention has been focused on the parabolic intervals 
$ [u,w]^{H}$. As a consequence, the parabolic analogue of the Combinatorial Invariance Conjecture (Conjecture~\ref{comb-inv-con}) has been considered to be the following.
\begin{con}
\label{comb-inv-con-par-falsa}
The parabolic Kazhdan--Lusztig polynomial $P^{H,x}_{u,v}(q)$ depends only on the combinatorial structure of the parabolic Bruhat interval $[u,v]^H$. 
\end{con}
Conjecture  \ref{comb-inv-con-par-falsa} has  recently been shown to be false in the case $x=q$ by Mongelli \cite{Mon}, who provides   the following counterexample for the Coxeter system $(W,S)$ of type $F_4$. Let $S=\{s_1,s_2,s_3,s_4\}$ with $(s_1s_2)^3=e$,  $(s_2s_3)^4=e$, $(s_3s_4)^3=e$, and $(s_is_j)^2=e$ for the other values of $i$ and $j$. Consider the subset $H= \{s_1,s_2,s_3\}$ of $S$ and the elements $u=s_3s_1s_2s_3s_4$, $v=s_3s_4s_2s_3s_1s_2s_3s_4$, $x=s_2s_3s_4$, and $y=s_4s_3s_1s_2s_3s_4$ of $W^H$. Then the parabolic intervals $[u,v]^H$ and $[x,y]^H$ are isomorphic while $P_{u,v}^{H,q} (q)=q$ and $P_{x,y}^{H,q} (q)=0$. We point out that $P_{u,v}^{H,-1} (q)=q+1$ and $P_{x,y}^{H,-1} (q)=1$, so that Conjecture~\ref{comb-inv-con-par-falsa} is false also in the case $x=-1$.
In particular, there cannot be a general method to compute the parabolic Kazhdan--Lusztig polynomial $P_{u,w}^{H,x} (q)$ (or  $R$-polynomial $R_{u,w}^{H,x} (q)$) just from the isomorphism type of the parabolic interval $ [u,w]^{H}$.

Roughly speaking, the basic idea of the present work is that,   for a poset-theoretic approach, also the elements 
that are not minimal coset representatives carry some information, and what should be considered is not just the parabolic interval $ [u,w]^{H} $ but the complete interval $ [u,w]$, together with  the notion of how $ [u,w]^{H} $ embeds in $ [u,w]$. 
This is reflected, for instance, in the fact that the matchings for which  (\ref{calcola}) holds are the $H$-special matchings, which are special matchings of 
the complete interval with a good behaviour with respect to the elements 
in the parabolic interval.
Within this perspective, the right approach to the generalization of the Combinatorial 
Invariance Conjecture  to the parabolic setting would be studying to what extent the following conjecture is true.

\begin{con}
\label{comb-inv-con-parab}
Let $(W_1,S_1)$ and $(W_2,S_2)$ be two Coxeter systems, $H_1 \subseteq S_1$ and 
$H_2\subseteq S_2$. Let $u_1,v_1 \in W_1^{H_1}$ and $u_2,v_2 \in W_2^{H_2}$ be such that there exists 
a poset-isomorphism  from $[u_1,v_1]$ to $[u_2,v_2]$ which restricts to a 
poset-isomorphism from $[u_1,v_1]^{H_1}$ to $[u_2,v_2]^{H_2}$. Then $P_{u_1,v_1}^{H_1,x} (q)=P_{u_2,v_2}^{H_2,x} (q)$  (equivalently, $R_{u_1,v_1}^{H_1,x} (q)=R_{u_2,v_2}^{H_2,x} (q)$).
\end{con}

Evidently, Conjecture~\ref{comb-inv-con-parab} reduces to Conjecture~\ref{comb-inv-con} for $H_1=H_2=\emptyset$. Mongelli's  is not a counterexample to Conjecture~\ref{comb-inv-con-parab} since the two Bruhat intervals   $[u,v]$ and $[x,y]$ of the counterexample are not isomorphic. The results of this work imply that, for doubly laced and dihedral Coxeter groups,  Conjecture~\ref{comb-inv-con-parab} holds when $u_1$ and $u_2$ are the identity elements.

\section{Notation, definitions and preliminaries}

This section reviews the background material that is needed  in the rest of this work.
We  follow   \cite{BB} and \cite[Chapter 3]{StaEC1} for undefined notation and 
terminology concerning, respectively, Coxeter groups and  partially ordered sets.

\subsection{Coxeter groups}
Given a Coxeter system $(W,S)$, we denote the entries of its Coxeter matrix $M$  by $m(s,s')$, for all $(s,s')\in S\times S$.  As usual, we say that a Coxeter system  - or a Coxeter group, by abuse of language - is simply laced (respectively, doubly laced) if $m(s,s')\leq 3$ (respectively, $m(s,s')\leq 4$), for all $(s,s')\in S\times S$. We denote by $e$
the identity of $W$, and we let $T =
\{ w s w ^{-1} : w \in W, \; s \in S \}$ be the
set of {\em reflections} of $W$.

Given $w \in W$, we denote by $\ell
(w )$ the length of $w $ with respect to $S$,
and we let
$$
\begin{array}{lll}
  D_{R}(w )& =&  \{ s \in S : \;
\ell(w  s) < \ell(w ) \}, \\
D_{L}(w )& =& \{ s \in S: \; \ell(sw)<\ell(w)\}. 
\end{array}
$$
We call the elements of $ D_R(w )$ and $D_{L}(w )$, respectively, the {\em right descents} and the 
{\em left descents} of $w $.

We now recall a result due to Tits (see \cite{Tit} or \cite[Theorem~3.3.1]{BB}). Given $s,s'\in S$ such that $m(s,s')< \infty$, let $\alpha_{s,s'}$ denote the alternating word  $ss'ss'\ldots$ of length $m(s,s')$. Two expressions are said to be linked by a braid-move (respectively, a nil-move) if it is possible to obtain one from the other by replacing a factor $\alpha_{s,s'}$ 
by a factor $\alpha_{s',s}$ (respectively, by deleting a factor $ss$).
\begin{thm}[Word Property]
\label{tits}
Let $u\in W$. Then:
 \begin{itemize}
\item any two reduced expressions of $u$ are linked by a finite sequence of braid-moves; 
\item any expression of $u$ (not necessarily reduced) is linked to any reduced expression of $u$ by a finite sequence of braid-moves and nil-moves.  \end{itemize}
\end{thm}

The {\em Bruhat graph}
of $W$ (see \cite{Dye2}, or, e.g., \cite[\S 2.1]{BB} or \cite[\S 8.6]{Hum}) is the directed graph
having $W$ as vertex set and having a directed edge from $u$ to $v$ if and only if $u^{-1}v \in T$ and $\ell (u)<\ell (v)$.  The transitive closure of the Bruhat graph of $W$
is a partial order on $W$ that is usually called the {\em Bruhat order} 
(see, e.g.,  \cite[\S 2.1]{BB} or \cite[\S 5.9]{Hum}) and that we denote by $\leq$.
Throughout this work, we always assume that $W$, and its subsets, are partially ordered by $\leq$.  
There is a well known characterization of Bruhat order
on a Coxeter group (usually referred to as the {\em Subword Property})
that we will use repeatedly in this work,
often without explicit mention. We recall it here for the reader's
convenience (a proof of it can be found, e.g., in \cite[\S 2.2]{BB} or \cite[\S 5.10]{Hum}).
By a {\em subword} of a word $s_{1}s_{2} \cdots s_{q}$ we mean
a word of the form
$s_{i_{1}}s_{i_{2}} \cdots s_{i_{k}}$, where $1 \leq i_{1}< \cdots
< i_{k} \leq q$.
\begin{thm}[Subword Property]
\label{subword}
Let $u,w \in W$. Then the following are equivalent:
\begin{itemize}
\item $u \leq w$ in the Bruhat order,
\item  every reduced expression for $w$ has a subword that is 
a reduced expression for $u$,
\item there exists a  reduced expression for $w$ having a subword that is 
a reduced expression for $u$.
\end{itemize}
\end{thm}
The Coxeter group $W$, partially ordered by Bruhat order, is a graded poset
having $\ell$ as its rank function.

For each subset $J\subseteq S$, we denote by  $W_J $ the parabolic subgroup of $W$ generated by $J$, and by  $W^{J}$ the set of minimal coset representatives:
$$W^{J} =\{ w \in W \, : \; D_{R}(w)\subseteq S\setminus J \}.$$
The following result is well known and a proof of it can be found, e.g., in \cite[\S 2.4]{BB} or \cite[\S 1.10]{Hum}.
\begin{pro}
\label{fattorizzo}
Let $J \subseteq S$. Then:
\begin{enumerate}
\item[(i)] 
every $w \in W$ has a unique factorization $w=w^{J} \cdot w_{J}$ 
with $w^{J} \in W^{J}$ and $w_{J} \in W_{J}$;
\item[(ii)] for this factorization, $\ell(w)=\ell(w^{J})+\ell(w_{J})$.
\end{enumerate}
\end{pro}
There are, of course, left versions of the above definition and
result. Namely, if we let  
\begin{equation}
^{J} \! W = \{ w \in W \, : \; D_{L}(w)\subseteq S\setminus J \} =(W^{J})^{-1}, 
\end{equation}
then every $w \in W$ can be uniquely factorized $
w=\, _{J} w\,  \cdot \,  ^{J}\! w$, where $_{J} w \in W_{J}$, $^{J} \!
w \in \,  ^{J} W$,
and $\ell(w)=\ell(_{J} w )+\ell(^{J} \! w)$.

The following is a well known result (see, e.g., \cite[Lemma~7]{Hom74}).
\begin{pro}
\label{unicomax}
Let \( J\subseteq S \) and $w \in W$.  The set 
$ W_{J}\cap [e,w] $ has a unique  maximal element  \( w_0(J) \),
so that \( W_{J}\cap [e,w]$ is the interval $[e,w_0(J)] \). 
\end{pro}

\subsection{Special matchings}
Given $x,y$ in a partially ordered set $ P$, 
we say that $y$ {\em covers} $x$ and we write $x \lhd y$ if the interval $[x,y]$ coincides with $\{x,y\}$.
An element $z \in [x,y]$ is said to be an {\em atom} (respectively, a {\em coatom})
of $[x,y]$ if $x \lhd z$ (respectively, $z \lhd y$). We say that a poset $P$ is {\em graded} if $P$ has a minimum and there is a function
$\rho : P \rightarrow {\mathbb N}$ (the {\em rank function}
of $P$) such that $\rho (\hat{0})=0$ and $\rho (y) =\rho (x)
+1$ for all $x,y \in P$ with $x \lhd y$. 
(This definition is slightly different from the one given in \cite{StaEC1}, but is
more convenient for our purposes.) 
The {\em Hasse diagram} of $P$ is the graph having $P$ as vertex set and $ \{ \{ x,y \} \in \binom {P}{2} : \text{ either $x \lhd y$
 or $y \lhd x$} \}$ as edge set.

A {\em matching} of a poset \( P \) is an involution
\( M:P\rightarrow P \) such that \( \{v,M(v)\}\) is an edge in the Hasse diagram of $P$, for all \( v\in V \).
A matching \( M \) of \( P \) is {\em special} if\[
u\lhd v\Longrightarrow M(u)\leq M(v),\]
 for all \( u,v\in P \) such that \( M(u)\neq v \).

The two simple results in the following lemma will be often used without explicit mention (see  \cite[Lemmas~2.1 and 4.1]{BCM1}). Given a poset $P$, two matchings $M$ and $N$ of  $P$, and $u\in P$, we denote by  $\langle M, N \rangle (u) $ the orbit of $u$ under the action of the subgroup of the symmetric group on $P$ generated by $M$ and $N$. We call an interval $[u, v]$ in a poset $P$ {\em dihedral} if it is isomorphic to a finite Coxeter system of rank 2 ordered
by Bruhat order.
\begin{lem}
\label{res}
Let \( P \) be a graded poset. 
\begin{enumerate}
\item Let $M$ be a special matching of $P$, and $u,v\in P$ be such that 
$M(v)\lhd v$ and $M(u)\rhd u$. Then $M$ restricts to a special matching of the interval $[u,v]$.
\item Let  $M$ and $N$ be two special matchings of $P$. Then, for all $u \in P$, the orbit $\langle M, N \rangle (u) $ is a dihedral interval.
\end{enumerate} 
\end{lem}

\subsection{Special matchings in Coxeter groups}

Let $(W,S)$ be a Coxeter system and recall that the Bruhat order is a  partial order on $W$.
 For $w\in W$, we say that $M$ is a matching of $w$ if $M$ is a matching of the lower Bruhat interval $[e,w]$.
If \( s\in D_{R}(w) \) (respectively, $s\in D_{L}(w)$) we define a matching 
\(\rho_{s}  \) (respectively, $\lambda_{s}$) of $w$ 
by \( \rho_{s}(u)=us \) (respectively, 
$\lambda_{s}(u)=su$) for all \( u\leq w \). From 
the ``Lifting Property'' (see, e.g., \cite[Theorem~1.1]{Deo77},  \cite[Proposition~2.2.7]{BB} or \cite[Proposition~5.9]{Hum}), it easily follows that $\rho_s$ 
(respectively, $\lambda_s$) is a special matching of $w$. 
We call a matching $M$ of $w$ a \emph{left multiplication
matching} if there exists \( s \in S \) such that  \( M=\lambda _{s} \)
on $[e,w]$, and we call it a  \emph{right multiplication
matching} if there exists \( s \in S \) such that  \( M=\rho _{s} \)
on $[e,w]$.

We recall the following  result, which will be needed in the proof of the main result of this work (see \cite[Lemma~4.3]{BCM1} for a proof).
\begin{pro}
\label{4.3}
Given a Coxeter system  $(W,S)$ and an element  $w\in W$, let $M$ and $N$ be two special matchings of $w$ such that $M(e)=s$ and $N(e)=t$, with $s \neq t$. Let $u \leq w$. Then the lower dihedral interval $[e,w]\cap W_{\{s,t\}}$ contains  an orbit of $\langle M,N \rangle$ having the same cardinality as the orbit  $\langle M,N\rangle(u)$ of $u$.
\end{pro}

Let $(W,S)$ be a Coxeter system and $w\in W$.
By Proposition~\ref{unicomax},  the intersection of the lower Bruhat interval $[e,w]$ with the dihedral parabolic subgroup $W_{\{s,t\}}$ 
generated by any two given generators $s,t\in S$ has a maximal element; for short, we denote it by $w_0(s,t)$ instead of $w_0(\{s,t\})$. 

We give the following symmetric definitions.
\begin{defi}
A  right system for $w$ is a quadruple $(J,s,t,M_{st})$ such that:
\begin{enumerate}

\item[R1.] $J\subseteq S$, $s\in J$, $t\in S\setminus J$, and $M_{st}$ is a special matching of $w_0(s,t)$ such that  $M_{st}(e)=s$ and  $M_{st}(t)=ts$;

\item[R2.] $(u^{J})^{\{s,t\}}\, \cdot \, M_{st} \Big ((u^{J})_{\{s,t\}} \, \cdot \, _{\{s\}} (u_{J})\Big )\, \cdot \,
 ^{\{s\}}(u_{J}) \leq w$,  for all $u\leq w$;

 \item[R3.] 
if $r\in J$ and $r \leq w^J$, then $r$ and $s$ commute;

\item[R4.]
\label{ddddxxxx}
\begin{enumerate}
\item  if $s\leq (w^J)^{\{s,t\}}$ and  $t\leq (w^J)^{\{s,t\}}$, then $M_{st}= \rho_s$,
\item   if $s\leq (w^J)^{\{s,t\}}$ and  $t\not \leq (w^J)^{\{s,t\}}$, then $M_{st}$ commutes with $\lambda_s$,
\item   if $s\not\leq (w^J)^{\{s,t\}}$ and  $t \leq (w^J)^{\{s,t\}}$, then $M_{st}$ commutes with $\lambda_t$;
\end{enumerate}

\item[R5.] if $v\leq w$ and $s\leq \, ^{\{s\}} (v_{J})$, then $M_{st}$ commutes with $\rho_s$ on $[e,v]\cap  [e,w_0(s,t)]=[e,v_0(s,t)]$.
\label{pure}

\end{enumerate} 

\end{defi}

\begin{defi}
A left system for $w$ is a quadruple $(J,s,t,M_{st})$ such that:

\begin{enumerate}
\item[L1.] $J\subseteq S$, $s\in J$, $t\in S\setminus J$, and $M_{st}$  is a special matching of $w_0(s,t)$ such that $M_{st}(e)=s$ and  $M_{st}(t)=st$;

\item[L2.] $    (_Ju)^{\{s\}} \, \cdot \, M_{st} \Big( \, (_Ju)_{\{s\}}  \,  \cdot  \,  _{\{s,t\}} (^J  u) \Big) \, 
\cdot \, ^{\{s,t\}}(^J u)    \leq w$,  for all $u\leq w$;
 \item[L3.] 
if $r\in J$ and $r \leq \, ^Jw$, then $r$ and $s$ commute;

\item[L4.]
\label{puresx}
\begin{enumerate}
\item  if $s\leq \, ^{\{s,t\}}(^Jw)$ and  $t\leq \, ^{\{s,t\}}(^Jw)$, then $M_{st}= \lambda_s$,
\item   if $s\leq \, ^{\{s,t\}}(^Jw)$ and  $t\not \leq \, ^{\{s,t\}}(^Jw)$, then $M_{st}$ commutes with $\rho_s$,
\item   if $s\not\leq \, ^{\{s,t\}}(^Jw)$ and  $t \leq \, ^{\{s,t\}}(^Jw)$, then $M_{st}$ commutes with $\rho_t$;
\end{enumerate}

\item[L5.] if $v\leq w$ and $s\leq  (_{J}v)^{\{s\}}$, then $M_{st}$ commutes with $\lambda_s$ on $[e,v]\cap  [e,w_0(s,t)]=[e,v_0(s,t)]$.
\label{sxsx}

\end{enumerate} 
\end{defi}

Given a right system for $w$, the matching $M$ associated with it is the matching of $w$ acting in the following way: 
for all $u\leq w$, 
$$M(u) = (u^{J})^{\{s,t\}}\, \cdot \, M_{st} \Big( (u^{J})_{\{s,t\}} \,  \cdot  \,  _{\{s\}} (u_{J}) \Big) \, 
\cdot \, ^{\{s\}}(u_{J}).$$

Note that $M$ acts as $\lambda_s$ on $[e, w_0(s,r)]$ for all $r\in J$, and as $\rho_s$ on $[e, w_0(s,r)]$ for all $r\in S\setminus (J\cup \{t\})$; moreover, if $s\in D_R(w)$, for the trivial choises $J= \{s\}$ and $M_{st}= \rho_s$, we obtain right multiplication matchings ($M = \rho_s$ on the entire interval $[e,w]$).

Symmetrically, given a  left system for $w$, the matching $M$ associated with it is the matching of $w$ acting in the following way: 
for all $u\leq w$,
 $$M(u) = (_Ju)^{\{s\}} \, \cdot \, M_{st} \Big( \, (_Ju)_{\{s\}}  \,  \cdot  \,  _{\{s,t\}} (^J  u) \Big) \, 
\cdot \, ^{\{s,t\}}(^J u).$$
We obtain left multiplication matchings as special cases.

We comment that distinct systems for $w$ might give rise to the same matching of $w$.
 
The following result is needed in the proofs of the main results of this work (see \cite{Mchara}).
\begin{thm}
\label{caratteri}
Let $w$ be any element of any arbitrary Coxeter group $ W$ and $M$ be a special matching of $w$. Then $M$ is  associated with a right or a left system of $w$.  
\end{thm}

\subsection{Kazhdan--Lusztig polynomials}
In introducing the (ordinary and parabolic) $R$-polynomials and Kazhdan--Lusztig polynomials, among all the equivalent definitions, we choose the combinatorial ones, since they suit our purposes best.

Given a Coxeter system $(W,S)$ and $H \subseteq S$,  we consider the set of minimal coset representatives  $W^{H}$ as a poset with the partial ordering induced by the Bruhat order on $W$. Given $u,v \in W^{H}$, $u \leq v$, we let
\[ [u,v]^{H} = \{ z \in W^{H}: \;  u \leq z \leq v \} ,\] be the (parabolic)  interval  in  
$W^{H}$ with bottom element $u$ and top element $v$.

The following two results are due to Deodhar, and we refer to
\cite[\S \S 2-3]{Deo87} for their proofs.
\begin{thm}
\label{7.1}
Let $(W,S)$ be a Coxeter system, and $H \subseteq S$.
 Then, for each $x \in \{ -1,q \}$,
  there is a unique family of polynomials $\{ R_{u,v}^{H,x}(q)
\} _{u,v \in W^{H}} \subseteq {\bf Z}[q]$ such that, for all $u,v \in W^{H}$:
\begin{enumerate}
\item $R^{H,x}_{u,v}(q)=0$ if $u \not \leq v$;
\item $R^{H,x}_{u,u}(q)=1$;
\item if $u<v$ and $s \in D_{L}(v)$, then
\[ R_{u,v}^{H,x} (q)= \left\{ \begin{array}{ll}
R_{su,sv}^{H,x}(q), & \mbox{if $s \in D_{L}(u)$,} \\
(q-1)R_{u,sv}^{H,x}(q)+qR_{su,sv}^{H,x}(q), & \mbox{if $s \notin D_{L}(u)$
and $su \in W^{H}$,} \\
(q-1-x)R_{u,sv}^{H,x}(q), & \mbox{if $s \notin D_{L}(u)$ and $su \notin
W^{H}$.}
\end{array} \right. \]
\end{enumerate}
\end{thm}
In the sequel, we will  often use the inductive formula of Theorem~\ref{7.1} without explicit mention.
\begin{thm}
\label{7.2}
Let $(W,S)$ be a Coxeter system, and $H \subseteq S$.
Then, for each $x \in \{ -1,q \}$,
 there is a unique family of polynomials $\{ P^{H,x}_{u,v}(q) \} 
_{u,v \in  W^{H}}  \subseteq {\bf Z} [q]$, such that, for all 
$u,v \in W^{H}$:
\begin{enumerate}
\item $P^{H,x}_{u,v}(q)=0$ if $u \not \leq v$;
\item $P^{H,x}_{u,u}(q)=1$;
\item deg$(P^{H,x}_{u,v}(q)) \leq  \frac{1}{2}\left(
\ell(v)-\ell(u)-1 \right) $, if $u < v$;
\item 
\( q^{\ell(v)-\ell(u)} \, P^{H,x}_{u,v} \left( \frac{1}{q} \right)
= \sum _{z\in[u,v]_H}   R^{H,x}_{u,z}(q) \,
P^{H,x}_{z,v}(q). \)
\end{enumerate}
\end{thm}

The polynomials $R^{H,x}_{u,v}(q)$ and $P^{H,x}_{u,v}(q)$ 
are called  the {\em
parabolic $R$-polynomials} and {\em parabolic  Kazhdan--Lusztig polynomials} of $W^{H}$ of type $x$. 
For $H=\emptyset$, $R_{u,v}^{\emptyset ,-1}(q)$ ($=R_{u,v}^{\emptyset ,q}(q)$) and $
P_{u,v}^{\emptyset ,-1}(q)$ ($=P_{u,v}^{\emptyset ,q}(q)$)
 are the ordinary $R$-polynomials $R_{u,v}(q)$ and Kazhdan--Lusztig polynomials $P_{u,v}(q)$  of $W$.
Another relationship between the parabolic Kazhdan--Lusztig polynomials and their
ordinary counterparts is established by the following result (see \cite[Proposition 3.4, and Remark 3.8]{Deo87}).
\begin{pro}
\label{7.3}
Let $(W,S)$ be a Coxeter system, $H \subseteq S$, and $u,v \in W^{H}$.
Then we have that
\[ P_{u,v}^{H,q}(q)=\sum _{w \in W_{H}}(-1)^{l(w)}P_{uw,v}(q) .\]
Furthermore, if $W_{H}$ is finite, then
\[ P_{u,v}^{H,-1}(q)=P_{uw_{0}^{H},vw_{0}^{H}}(q) , \]
where $w_{0}^{H}$ is the longest element of $W_H$. 
\end{pro}

We refer to \cite{BB}, \cite{Deo87}, and \cite{Hum}  for more details concerning general Coxeter group theory and parabolic Kazhdan--Lusztig polynomials.

\section{Commuting special matchings} 

In this section, we prove existence results for special matchings commuting with a given one. These results are needed in Section \ref{parabolici}.

We will make repeated use of the following easy result. 
\begin{lem}
\label{commutano}
Let $(W,S)$ be any arbitrary Coxeter system.
Two special matchings $M$ and $N$ of $w\in W$ commute if and only if they commute on the lower dihedral intervals  containing $M(e)$ and $N(e)$ (in particular, on $[e,w_0(s,t)]$ if $s=M(e), t= N(e), s \neq t$).
\end{lem}
\begin{proof}
The result follows directly from Proposition~\ref{4.3}.
\end{proof}
We recall that a special matching $M$ of $w$ stabilizes the intersection of $[e,w]$ with any parabolic subgroup containing $M(e)$ (see \cite[Proposition~5.3]{BCM1}). We will use this fact without explicit mention.

\subsection{Right (resp. left) systems and left (resp. right) multiplication matchings}
\begin{pro}
\label{simply-lambda}
Let $W$ be a doubly laced Coxeter group, $w \in W$ and $M$ a special matching of $w$ associated with a right system $(J,s,t,M_{st})$, $M$ not a left  multiplication matching. Then there exists a left multiplication matching $\lambda$ of $w$ commuting with $M$ such that $\lambda(w)\neq M(w)$, unless $m(s,t)=4$ and we are in one of the following cases:\\
{\bf First case:}
\begin{enumerate}
\item $w_0(s,t)=tst $,   
\item $(w^J)^{\{s,t\}}= e$, 
\item $w=tst \cdot \, ^{\{s\}}(w_{J})$,
\end{enumerate}
{\bf Second case:}
\begin{enumerate}
\item  $w_0(s,t)=tsts=stst$,    
\item $(w^J)^{\{s,t\}}= e$,
\item either 
\begin{itemize}
\item $w=tsts \cdot \, ^{\{s\}}(w_{J})$, or 
\item $w=tst \cdot \, ^{\{s\}}(w_{J})$ with $s\leq \, ^{\{s\}}(w_{J})$,
\end{itemize}
\item $M(tsts)=sts$ (and then $M_{st}$ must be the matching mapping $e$ to $s$, $t$ to $ts$, $st$ to $tst$, and $sts$ to $tsts$).
\end{enumerate}
\end{pro}
\begin{proof}
If $(w^J)^{\{s,t\}}\neq e$, there exists $l\in D_L((w^J)^{\{s,t\}})$. Then $l\in D_L(w)$ and $\lambda_l$ is a special matching of $w$ which satisfies  $M(w) \neq \lambda_l(w)$ since 
 $$M(w) = (w^{J})^{\{s,t\}}\, \cdot \, M_{st} \Big( (w^{J})_{\{s,t\}} \,  \cdot  \,  _{\{s\}} (w_{J}) \Big) \, 
\cdot \, ^{\{s\}}(w_{J})$$
while
$$\lambda_l(w) = l \cdot (w^{J})^{\{s,t\}}\, \cdot \,  \Big( (w^{J})_{\{s,t\}} \,  \cdot  \,  _{\{s\}} (w_{J}) \Big) \, 
\cdot \, ^{\{s\}}(w_{J}).$$
We have to show that $M$ and $\lambda_l$ commute. We distinguish the following cases, in which we apply Lemma~\ref{commutano}.\\
(a) $l\notin \{s,t\}$\\
By Property R3 of the definition of a right system, either $l \notin J$ or $l$ commutes with $s$. In the first case, $M$ acts as $\rho_s$ on  $[e,w_0(s,l)]$  and hence commutes with $\lambda_l$. In the second case, $M$ and $\lambda_l$ clearly commutes on  $[e,w_0(s,l)]$ because  $[e,w_0(s,l)]$   has just  4 elements.\\
(b) $l=t$\\
By Property R4, $M$ commutes with $\lambda_t$ on the lower interval  $[e,w_0(s,t)]$ (we are either in case (a) or in case (c) of  Property R4).\\
(c) $l=s$\\
By  Lemma~\ref{commutano}, we need to show that $M$ and $\lambda_s$ commute on every lower dihedral intervals $[e,w_0(s,r)]$, with $r\in S\setminus \{s\}$. For $r=t$, it follows from  Property R4 (we are either in case (a) or in case (b) of  Property R4). For $r\neq t$, $M$ acts on $[e,w_0(s,r)]$ as $\rho_s$ or $\lambda_s$, and in both cases $M$ commutes with $\lambda_s$ on $[e,w_0(s,r)]$.

We now suppose $(w^J)^{\{s,t\}}= e$. By the definition of a right system, $M$ acts as $\lambda_s$ on every lower dihedral interval  $[e,w_0(s,r)]$, $r \in S \setminus \{s,t\}$. On $[e,w_0(s,t)]$, $M$ does not act as $\lambda_s$ as otherwise $M$ would coincide with $\lambda_s$  everywhere, but $M$ is not a left multiplication matching by hypothesis. In particular, $M(t)= ts \neq st$. If $w_0(s,t)=ts$,  then either $w= ts \cdot  \, ^{\{s\}}(w_{J})$, or $w= t \cdot  \, ^{\{s\}}(w_{J})$ with $s \leq  \, ^{\{s\}}(w_{J})$. Indeed, the second case cannot occur since $M(w)$ would be $M(t)  \cdot  \, ^{\{s\}}(w_{J})= ts \cdot  \, ^{\{s\}}(w_{J})$  and we would have $M(w)\rhd w$, which is impossible.  So $w= ts \cdot  \, ^{\{s\}}(w_{J})$, $t\in D_L(w)$, $\lambda_t$ is a special matching of $w$, and $M$ commutes with $\lambda_t$ by Lemma~\ref{commutano}. Moreover, 
$$M(w)= M(ts) \cdot  \, ^{\{s\}}(w_{J})= t \cdot  \, ^{\{s\}}(w_{J}),$$
while 
$$\lambda_t(w)= s  \cdot  \, ^{\{s\}}(w_{J}),$$
so $M(w) \neq \lambda_s(w)$.

So we may assume that the lower dihedral interval $[e,w_0(s,t)]$  has at least 6 elements (hence it has  6 or 8 elements, since $W$ is doubly laced). 
Suppose that $[e,w_0(s,t)]$ has 6 elements.
Necessarily, $M(e)=s$, $M(t)=ts$, and $M(st)= w_0(s,t)\in \{tst,sts\}$.
If $w_0(s,t)=tst\neq sts$, then, since $tst \leq w$ and $t \not \leq \, ^{\{s\}}(w_{J})$, we have $w= tst \cdot  \, ^{\{s\}}(w_{J})$ and we are in the first case of the statement of the proposition.
 If $w_0(s,t)=sts$, then,
since $sts \leq w$ and $t \not \leq \, ^{\{s\}}(w_{J})$, we have only two possibilities: either $w= sts \cdot  \, ^{\{s\}}(w_{J})$, or $w= st \cdot  \, ^{\{s\}}(w_{J})$ with $s \leq  \, ^{\{s\}}(w_{J})$. Indeed, the second one cannot occur since, in that case, $M(w)= M(st)  \cdot  \, ^{\{s\}}(w_{J})= sts \cdot  \, ^{\{s\}}(w_{J})$  and we would have $M(w)\rhd w$, which is impossible. So $w= sts \cdot  \, ^{\{s\}}(w_{J})$; thus $s\in D_L(w)$, $\lambda_s$ is a special matching of $w$, and $M$ commutes with $\lambda_s$ by Lemma~\ref{commutano}. Moreover, 
$$M(w)= M(sts) \cdot  \, ^{\{s\}}(w_{J})= st \cdot  \, ^{\{s\}}(w_{J}),$$
while 
$$\lambda_s(w)= ts  \cdot  \, ^{\{s\}}(w_{J}),$$
so $M(w) \neq \lambda_s(w)$.

Suppose that $[e,w_0(s,t)]$ has 8 elements, i.e.,  $w_0(s,t)=stst=tsts$ since $W$ is doubly laced. 
Then, since $tsts \leq w$ and  $t \not \leq  \, ^{\{s\}}(w_{J})$, we have either $w=tsts \; ^{\{s\}}(w_{J})$, or $w=tst \; ^{\{s\}}(w_{J})$ with $s\leq \, ^{\{s\}}(w_{J})$. In both cases $t\in D_L(w)$, hence $\lambda_t$ is a special matching of $w$. The matching $\lambda_t$ always commutes with $M$ since $M(t)=ts$ by the definition of a right system; moreover, evidently, $\lambda_t(w)=M(w)$  if and only if  $M(tsts)=sts$. 

The proof is complete.
\end{proof}

Note that the conditions in the cases of the statement of Proposition~\ref{simply-lambda} (and also of the forthcoming Propositions~\ref{simply-lambda2},  \ref{double-rho}, and  \ref{double-rho2}) are redundant but we prefer to emphasize them since they are needed later.

The symmetric version of Proposition~\ref{simply-lambda} is the following.
\begin{pro}
\label{simply-lambda2}
Let $W$ be a doubly laced Coxeter group, $w \in W$ and $M$ a special matching of $w$ associated with a left system $(J,s,t,M_{st})$, $M$ not a right multiplication matching. Then there exists a right multiplication matching $\rho$ of $w$ commuting with $M$ such that $\rho(w) \neq M(w)$, unless $m(s,t)=4$ and we are in one of the following cases:\\
{\bf First case:}
\begin{enumerate}
\item $w_0(s,t)=tst$,    
\item $^{\{s,t\}} (^Jw)= e$, 
\item $w=   (_{J}w)^{\{s\}} \, \cdot tst $,
\end{enumerate}
{\bf Second case:}
\begin{enumerate}
\item  $w_0(s,t)=stst=tsts$,    
\item $^{\{s,t\}} (^Jw)= e$,
\item either 
\begin{itemize}
\item $w= (_{J}w)^{\{s\}} \, \cdot stst    $, or 
\item $w=  (_{J}w)^{\{s\}} \, \cdot tst $ with $s\leq  (_{J}w)^{\{s\}}$,
\end{itemize}
\item $M(stst)=sts$ (and then $M_{st}$ must be the matching mapping $e$ to $s$, $t$ to $st$, $ts$ to $tst$, and $sts$ to $stst$).
\end{enumerate}
\end{pro}

\begin{rem}
\label{osservazioni}
We make the following observations.
\begin{enumerate}
\item In the proof of Proposition~\ref{simply-lambda}, we use the hypothesis that $W$ be doubly laced only in the case  $(w^J)^{\{s,t\}}= e$ and we actually use only the fact that $m(s,t)\leq 4$.

\item Proposition~\ref{simply-lambda} does not hold in the special cases we excluded.  A trivial example can be found in the dihedral Coxeter system $(W, \{s,t\})$ with $m(s,t)=4$. Since usually  dihedral Coxeter systems can be treated separately, we give also the following less trivial counterexample which explains better what obstructions may occur. Let $(W,\{s,t,r\})$ be the Coxeter system with Coxeter matrix satisfying $m(s,t)=4, m(s,r)=m(t,r)=3$. Consider the element $w=tstrs \in W$ and the matching associated with the right system $(J=\{s,r\}, s,t, M_{st})$ with $M_{st}(st)=tst$ and  $M_{st}(sts)=tsts$. Then $M$ commutes with the unique left multiplication special matching $\lambda_t$ but it coincides with it on $w$.

\end{enumerate}

These observations similarly hold true, mutatis mutandis, also for Proposition \ref{simply-lambda2}.
\end{rem}

\subsection{Right (resp. left) systems and right (resp. left) multiplication matchings}
\begin{pro}
\label{double-rho}
Let $W$ be a doubly laced Coxeter group, $w \in W$ and $M$ a special matching of $w$ associated with a right system  $(J,s,t,M_{st})$, $M$ not a right multiplication matching. Then there exists a right multiplication matching $\rho$ of $w$ commuting with $M$ such that $\rho(w) \neq M(w)$,  unless $m(s,t)=4$ and we are in the following case:
\begin{enumerate}
\item $w_0(s,t)=tst $,    
\item $w_J= e$,
\item $w=  (w^{J})^{\{s,t\}} \cdot tst$.
\end{enumerate}
\end{pro}
\begin{proof}
If $^{\{s\}}(w_J)\neq e$, there exists $r\in D_R(^{\{s\}}(w_J))$. 
Then $r\in D_R(w)$ and $\rho_r$ is a special matching of $w$ which satisfies $M(w) \neq \rho_r(w)$ since 
 $$M(w) = (w^{J})^{\{s,t\}}\, \cdot \, M_{st} \Big( (w^{J})_{\{s,t\}} \,  \cdot  \,  _{\{s\}} (w_{J}) \Big) \, 
\cdot \, ^{\{s\}}(w_{J})$$
while
$$\rho_r(w) =  (w^{J})^{\{s,t\}}\, \cdot \,  \Big( (w^{J})_{\{s,t\}} \,  \cdot  \,  _{\{s\}} (w_{J}) \Big) \, 
\cdot \, ^{\{s\}}(w_{J}) \cdot r.$$
In order to show that $M$ and $\rho_r$ commute,  we apply Lemma~\ref{commutano}. If $r \neq s$, $M$ acts as $\lambda_s$ on the lower dihedral interval  $[e,w_0(s,r)]$ and hence it commutes with $\rho_r$. If $r=s$, we need to show that $M$ and $\rho_s$ commute on every  lower dihedral interval  $[e,w_0(s,l)]$, with $l\in S\setminus \{s\}$. For $l=t$, it follows from  Property R5. For $l\neq t$, $M$ acts on $[e,w_0(s,l)]$ either as $\lambda_s$ or as $\rho_s$: in both cases it commutes with $\rho_s$ on $[e,w_0(s,l)]$.

We now suppose $^{\{s\}}(w_J)= e$. By the definition of a right system, $M$ acts as $\rho_s$ on every lower dihedral interval $[e,w_0(s,r)]$, $r \in S \setminus \{s,t\}$. On $[e,w_0(s,t)]$, $M$ does not act as $\rho_s$ as otherwise $M$  would coincide with $\rho_s$ everywhere, but $M$ is not a right multiplication matching by hypothesis. 
Since $M(e)=\rho_s(e)=s$ and  $M(t)=\rho_s(t)=ts$, $w_0(s,t)$ cannot be $sts$, which implies that $m(s,t)=4$ (since $W$ is a doubly laced Coxeter group). If $w_0(s,t)=tst$, then also  $_{\{s\}}(w_J)= e$, as otherwise $tsts$ would be $\leq w$ by the Subword Property. Hence $w_J = e $ and $w=  (w^{J})^{\{s,t\}} \cdot   (w^{J})_{\{s,t\}}$. Since $M(w)\lhd w$ and $M(w)= (w^{J})^{\{s,t\}} \cdot  M( (w^{J})_{\{s,t\}})$ by definition, we have $M( (w^{J})_{\{s,t\}})\lhd (w^{J})_{\{s,t\}}$ and hence
$ (w^{J})_{\{s,t\}}\notin \{e, t, st\}$. Moreover, since $s\in J$, clearly $s\notin D_R(w^{J})$, which implies $s\notin D_R( (w^{J})_{\{s,t\}})$, and thus $(w^{J})_{\{s,t\}}\notin  \{s, ts\} $. The only possibility left is  $(w^{J})_{\{s,t\}}= tst$. Hence $w=  (w^{J})^{\{s,t\}} \cdot tst$ and
 we are in the case we excluded in the statement of the proposition. 

Now suppose that $w_0(s,t)= stst=tsts $; since $M\neq \rho_s$, necessarily $M(st)=tst $ and $M(sts) = stst $.   
By Property R4, (a), it is not possible that both $s$ and $t$ are $\leq (w^{J})^{\{s,t\}}$.
Since $M$ does not commute with $\lambda_s$ on $[e,w_0(s,t)]$, we have $s \not \leq (w^{J})^{\{s,t\}}$ by Property~R4,~(b). 
  Since  $stst \leq w$, we have that $(w^{J})_{\{s,t\}} \,  \cdot  \,  _{\{s\}} (w_{J}) $  can be either $stst$ or $sts$. The second case is impossible since $M(w)= (w^{J})^{\{s,t\}} \cdot M(sts)=(w^{J})^{\{s,t\}} \cdot  stst$ would be greater than $w$. Hence 
$$w= (w^{J})^{\{s,t\}} \cdot  stst,$$
$s\in D_R(w)$, and $\rho_s$ is a special matching of $w$. Since $M$ and $\rho_s$ commute on every lower dihedral interval, they commute everywhere by Lemma~\ref{commutano}. Moreover $M(w) \neq \rho_s(w)$ since
$$M(w)= M \Big( (w^{J})^{\{s,t\}}\, \cdot \,  stst  \Big)=   (w^{J})^{\{s,t\}}\, \cdot \, M(stst)=  (w^{J})^{\{s,t\}}\, \cdot \, sts$$ 
while 
$$\rho_s(w)= \rho_s \Big( (w^{J})^{\{s,t\}}\, \cdot \,  stst  \Big)= (w^{J})^{\{s,t\}}\, \cdot \, tst.$$ 
The proof is complete.
\end{proof}

The symmetric version of Proposition~\ref{double-rho} is the following.
\begin{pro}
\label{double-rho2}
Let $W$ be a doubly laced Coxeter group, $w \in W$ and $M$ a special matching of $w$ associated with a left system  $(J,s,t,M_{st})$, $M$ not a left  multiplication matching. Then there exists a left multiplication matching $\lambda$ of $w$ commuting with $M$ such that $\lambda(w)\neq M(w)$,  unless $m(s,t)=4$ and we are in the following case:
\begin{enumerate}
\item $w_0(s,t)=tst $,    
\item $_Jw= e$,
\item $w=  tst  \cdot \, ^{\{s,t\}} (^{J}w)$.
\end{enumerate}.
\end{pro}

\begin{rem}
We make the following observations.
\begin{enumerate}
\item In the proof of Proposition~\ref{double-rho}, we  use the hypothesis that $W$ be doubly laced only in the case  $^{\{s\}}(w_J)= e$ and we actually use only the fact that $m(s,t)\leq 4$.

\item Proposition~\ref{double-rho} does not hold if $m(s,t)=5$. A trivial counterexample can be found in the dihedral Coxeter system $(W, \{s,t\})$ with $m(s,t)=5$. Since usually  dihedral Coxeter systems can be treated separately, we give also the following less trivial counterexample which explains better what obstruction may occur. Consider the Coxeter system $(W,\{s,t,r\})$ with  $m(s,t)=5, m(s,r)=m(t,r)=3$, the element $w=trstst \in W$ and the matching associated with the right system $(J=\{s,r\}, s,t, M_{st})$ with $M_{st}(st)=tst$, $M_{st}(sts)=stst$, and $M_{st}(tsts)=tstst$. Then $M$ does not commute with the unique right multiplication special matching $\rho_t$.

\end{enumerate}
These observations similarly hold true, mutatis mutandis, also for Proposition \ref{double-rho2}.

\end{rem}

As an immediate consequence, we have the following result.
\begin{cor}
\label{semplicesemplice}
Let $W$ be a simply laced Coxeter group, $w \in W$, and $M$ a special matching of $w$. Then there exist a right multiplication matching  and a left multiplication matching  that commute with $M$ and such that do not agree with $M$ on $w$.
\end{cor}
\begin{proof}
The assertion follows from Theorem~\ref{caratteri} and Propositions~\ref{simply-lambda},  \ref{simply-lambda2}, \ref{double-rho},  \ref{double-rho2}.
\end{proof}

\section{Parabolic $R$-polynomials and $H$-special matchings}
\label{parabolici}

In this section, for all $H\subseteq S$ and $w\in W^H$, we give a method for computing the parabolic Kazhan--Lusztig $R$-polynomials $\{R_{u,w}^{H,x}(q)\}_{u\in W^H}$ from the only knowledge of: 
\begin{itemize}
\item the isomorphism type (as a poset) of the interval $[e,w]$,
\item  which  elements of the interval $[e,w]$ are  in $W^H$.
\end{itemize}
We prove this result  in the case $W$ is either a doubly laced Coxeter group (i.e., $m(r,r')\leq 4$, for all $r,r'\in S$), or
a dihedral Coxeter group (i.e., a Coxeter group of rank 2). In particular, the result holds for all Weyl groups. It is worth noting that this is also a result on parabolic Kazhdan--Lusztig polynomials since these are  equivalent to the parabolic $R$-polynomials.

\subsection{H-special matchings and calculating special matchings}
We now define the tools that compute the parabolic  Kazhan--Lusztig polynomials.

Let  $(W,S)$  be an arbitrary Coxeter system, $H \subseteq S$, and $w \in W^H$. An {\em $H$-special matching of $w$} is a special matching of $w$ such that 
$$u \leq w, u \in W^H, M(u) \lhd u \Rightarrow M(u) \in W^H.$$

Note that the $\emptyset$-special matchings are exactly the special matchings and that a left multiplication matching  is $H$-special 
for all $H \subseteq S$.

For convenience' sake, we say that an $H$-special matching $M$ of $w$ \emph{calculates} the parabolic Kazhdan--Lusztig $R$-polynomials (or  is \emph{calculating}, for short) if, for all $u \in W^H$, $u \leq w$, we have
\begin{equation}
 R_{u,w}^{H,x} (q)= \left\{ \begin{array}{ll}
R_{M(u),M(w)}^{H,x}(q), & \mbox{if $M(u)  \lhd u$,} \\
(q-1)R_{u,M(w)}^{H,x}(q)+qR_{M(u),M(w)}^{H,x}(q), & \mbox{if $M(u) \rhd u$
and $M(u) \in W^{H}$,} \\
(q-1-x)R_{u,M(w)}^{H,x}(q), & \mbox{if $M(u) \rhd u$ and  $M(u) \notin W^{H}$.} 
\end{array} \right. 
\end{equation}
For  this definition, it is essential that the special matching be $H$-special. Note that all left multiplication matchings are calculating. 

We want to show that all $H$-special matchings are calculating for all doubly laced Coxeter groups and all dihedral groups (indeed, we prove it for a larger class of situations). 

\subsection{Commuting matchings and calculating matchings}
We need the following easy lemma.
\begin{lem}
\label{coatomi-al-piu}
Let  $(W,S)$  be an arbitrary Coxeter system,  $H \subset S$, and $v \notin W^H$. 
Then at most one of the coatoms of $v$ belongs to $W^H$.
\end{lem}
\begin{proof}
Since $v \notin W^H$, the set $H \cap D_{R}$ is non-empty. Let $s \in H \cap D_{R}$. Then there exists $y \in W$ such that $v= y \cdot s $ and $\ell(v)= \ell (y) + 1$. By the Deletion Property, $s$ is a right descent of all coatoms of $v$ except $y$.  
\end{proof}

In the proof of the following theorem and in the sequel, we use the inductive formula of Theorem~\ref{7.1} without explicit mention.
\begin{thm}
\label{secommutano}
Given a  Coxeter system  \( (W,S) \) and $H \subseteq S$, let $w \in W^H$ and $M$ be an $H$-special matching of $w$. Suppose that
\begin{itemize}
\item every $H$-special matching of $v$ is calculating, for all  \( v \in W ^H\), $v < w$, 
\item there exists a calculating special matching $N$ of $w$ commuting with $M$ and such that $M(w) \neq N(w)$.
\end{itemize}
Then $M$  is calculating.
\end{thm}
\begin{proof}
We prove the claim by induction on $\ell (w)$, the result being clearly
true if $\ell (w) \leq 2$. So assume $\ell (w) \geq 3$, and let $u\in [e,w]$, $u\in W^H$. 
The orbit  of $w$ under the action of the group generated by $M$ and $N$ is $\langle N,M \rangle(w)=\{w, N(w), M(w), MN(w)=NM(w)\}$ and is contained in $W^H$ since both $N$ and $M$ are $H$-special. On the other hand, the orbit $\langle N,M \rangle( u)$  can have either cardinality 4 or cardinality 2, and can be contained in $W^H$ or not. We may assume that it is not contained in $W^H$ as otherwise we could prove that $M$ is calculating by the same arguments as in \cite[Theorem~7.8]{BCM1}. 

If the cardinality of  $\langle N,M \rangle( u)$ is $2$, then $u \lhd N(u)=M(u)\notin W^H$, since otherwise the orbit would be contained in  $W^H$  since $N$ and $M$ are $H$-special. In this case, by our induction hypothesis
\begin{eqnarray*}
R^{H,x}_{u,w} (q)&=& (q-1-x) R^{H,x}_{u,N(w)}= (q-1-x)^2 R^{H,x}_{u,MN(w)} \\
&=&   (q-1-x)^2 R^{H,x}_{u,NM(w)} = (q-1-x) R^{H,x}_{u,M(w)},\\
\end{eqnarray*}
as desired. (Here and in the sequel, we use the fact that $M$ restricts to a special matching of $N(w)$ and hence we can use the induction hypothesis since $\ell (N(w) ) < \ell (w)$.)

Now assume that  the cardinality of  $\langle N,M \rangle( u)$ is $4$. By Lemma~\ref{coatomi-al-piu} and the fact that $N$ and $M$ are $H$-special, there are 5 cases to be considered.\\
(a) $u \lhd N(u) \notin W^H$, $u \rhd M(u) \in W^H$, $NM(u)= MN(u) \notin W^H$. \\
By our induction hypothesis, 
\begin{eqnarray*}
R^{H,x}_{u,w}  (q)&= &(q-1-x) R^{H,x}_{u,N(w)}=   (q-1-x) R^{H,x}_{M(u),MN(w)}\\
&=  & (q-1-x) R^{H,x}_{M(u),NM(w)} =   R^{H,x}_{M(u),M(w)}.
\end{eqnarray*}
(b) $u \rhd N(u) \in W^H$, $u \lhd M(u) \notin W^H$, $NM(u)= MN(u) \notin W^H$. \\
By our induction hypothesis, 
\begin{eqnarray*}
R^{H,x}_{u,w} (q) &= & R^{H,x}_{N(u),N(w)}=   (q-1-x) R^{H,x}_{N(u),MN(w)} \\
&=  & (q-1-x) R^{H,x}_{N(u),NM(w)}  =  (q-1-x) R^{H,x}_{u,M(w)}.
\end{eqnarray*}
(c) $u \lhd N(u) \in W^H$, $u \lhd M(u) \notin W^H$, $NM(u)= MN(u) \notin W^H$. \\
By our induction hypothesis, 
\begin{eqnarray*}
R^{H,x}_{u,w}  (q) &= & (q-1) R^{H,x}_{u,N(w)} +  q R^{H,x}_{N(u),N(w)}\\
 &= &(q-1) (q-1-x) R^{H,x}_{u,MN(w)} +q  (q-1-x) R^{H,x}_{N(u),MN(w)} \\
&=  &(q-1-x)((q-1) R^{H,x}_{u,NM(w)} +q  R^{H,x}_{N(u),NM(w)}) \\
& = & (q-1-x) R^{H,x}_{u,M(w)}.
\end{eqnarray*}
(d) $u \lhd N(u) \notin W^H$, $u \lhd M(u) \in W^H$, $NM(u)= MN(u) \notin W^H$. \\
By our induction hypothesis, 
\begin{eqnarray*}
R^{H,x}_{u,w} (q) &= & (q-1-x) R^{H,x}_{u,N(w)} \\
&=  &(q-1-x)((q-1) R^{H,x}_{u,MN(w)} +q  R^{H,x}_{M(u),MN(w)}) \\
&=  &(q-1-x)((q-1) R^{H,x}_{u,NM(w)} +q  R^{H,x}_{M(u),NM(w)}) \\
&= & (q-1) R^{H,x}_{u,M(w)}  +q  R^{H,x}_{M(u),M(w)}.
\end{eqnarray*}
(e) $u \lhd N(u) \notin W^H$, $u \lhd M(u) \notin W^H$, $NM(u)= MN(u) \notin W^H$. \\
By our induction hypothesis, 
\begin{eqnarray*}
R^{H,x}_{u,w} (q) &= & (q-1-x) R^{H,x}_{u,N(w)} =  (q-1-x)^2 R^{H,x}_{u,MN(w)} \\
 &= & (q-1-x)^2 R^{H,x}_{u,NM(w)}=  (q-1-x) R^{H,x}_{u,M(w)}.
\end{eqnarray*}
The proof is complete.
\end{proof}

The following result is a special case of Theorem~\ref{doppiamente-allacciati}. Since it concerns an important class of Coxeter groups, including the symmetric group, and its proof is much simpler than the proof of Theorem~\ref{doppiamente-allacciati}, we give it here explicitely.
\begin{cor}
\label{semplicemente-allacciati}
Let \( (W,S) \) be a simply laced Coxeter system, $H$ be any arbitrary subset of $S$ and $ w$ be any arbitrary element  of $ W ^H$. Then every $H$-special matching $M$ of \( w \) calculates the $R^{H,x}$-polynomials.
\end{cor}
\begin{proof}
The assertion follows immediately by induction using Corollary~\ref{semplicesemplice} and Theorem~\ref{secommutano}, since  left multiplication matchings are calculating by definition.
\end{proof}

\subsection{Doubly laced Coxeter groups}
We now prove that every $H$-special matching calculates the $R^{H,x}$-polynomials of doubly laced Coxeter groups.
The following easy lemma is needed in the proof.
\begin{lem}
\label{calcoli}
Let $(W,S)$ be a Coxeter system, $s,t \in S$,  $m(s,t)\geq4$, $H \subseteq S$.   Let $w\in W$ be such that $t \not \leq w$ and $st   w  \in W^H$. Let $v \leq w$ be such that  $s\notin D_{L}(v)$ and $\{ v, sv, tv, stv\} \subseteq W^H$. Then
$$R^{H,x}_{sv,stw}= R^{H,x}_{tv,stw}. $$
\end{lem}
\begin{proof}
By hypothesis, $\ell (tsv)=\ell(stv)=\ell(sv)+1=\ell(tv)+1=\ell(v)+2$. 
We have
$$R^{H,x}_{sv,stw}=R^{H,x}_{v,tw}= (q-1) R^{H,x}_{v,w} + q  R^{H,x}_{tv,w}=  (q-1) R^{H,x}_{v,w}$$
 (where the last equality holds since $t \not \leq w$ and hence  $tv \not \leq w$), and

$$ R^{H,x}_{tv,stw}=  (q-1) R^{H,x}_{tv,tw} + q  R^{H,x}_{stv,tw} =   (q-1) R^{H,x}_{tv,tw} =   (q-1) R^{H,x}_{v,w}$$ (where the second equality holds since $t \not \leq w$ and hence  $stv \not \leq tw$).
\end{proof}

\begin{thm}
\label{doppiamente-allacciati}
Let \( (W,S) \) be a doubly laced Coxeter system, $H$ be any arbitrary subset of $S$ and $ w$ be any arbitrary element  of $ W ^H$. Then every $H$-special matching  of \( w \) calculates the $R^{H,x}$-polynomials.
\end{thm}
\begin{proof}
We proceed by induction on $\ell(w)$, the case $\ell(w)\leq 1$ being trivial.

Let $M$ be a  $H$-special matching  of \( w \). We assume that $M$ is not a left multipication matching, since left multiplication matchings are calculating by definition.

If there exists a left multiplication matching $\lambda$ commuting with $M$ and such that $\lambda(w) \neq M(w)$, then we can conclude by Theorem~\ref{secommutano}. 
We assume that such a left multiplication matching does not exist. By Theorem~\ref{caratteri} and Propositions~\ref{simply-lambda} and \ref{double-rho2}, necessarily $m(s,t)=4$ and
we are in one of the following cases:\\
{\bf Case 1:} $M$ is associated with a right system $(J,s,t,M_{st})$ and
\begin{enumerate}
\item $w_0(s,t)=tst $,   
\item $(w^J)^{\{s,t\}}= e$, 
\item $w=tst \cdot \, ^{\{s\}}(w_{J})$,
\end{enumerate}
{\bf Case 2:} $M$ is associated with a right system $(J,s,t,M_{st})$ and
\begin{enumerate}
\item  $w_0(s,t)=tsts=stst$,    
\item $(w^J)^{\{s,t\}}= e$,
\item either 
\begin{itemize}
\item $w=tsts \cdot \, ^{\{s\}}(w_{J})$, or 
\item $w=tst \cdot \, ^{\{s\}}(w_{J})$ with $s\leq \, ^{\{s\}}(w_{J})$,
\end{itemize}
\item $M(tsts)=sts$ (and then $M_{st}$ must be the matching mapping $e$ to $s$, $t$ to $ts$, $st$ to $tst$, and $sts$ to $tsts$).
\end{enumerate}
{\bf Case 3:} $M$ is associated with a left system $(J,s,t,M_{st})$ and
\begin{enumerate}
\item $w_0(s,t)=tst $,    
\item $_Jw= e$,
\item $w=  tst  \cdot \, ^{\{s,t\}} (^{J}w)$.
\end{enumerate}.

{\bf Case 1.} Notice that, in this case, $s \not\leq \, ^{\{s\}}(w_{J})$  (as otherwise $tsts$ would be $\leq w$ by the Subword Property), $ w_J= \, ^{\{s\}}(w_{J})$, $w=tst \cdot \, w_{J}$ and $M(w)= st \cdot \, w_{J}$.
Then $t\in D_L(w)$ and $\lambda_t$ is a special matching of $w$. $M$ and $\lambda_t$ commute but they agree on $w$ (i.e., $|\langle M, \lambda_t \rangle (w)| = 2$). 
 The lower dihedral interval $[e, w_0(s,t)]$ contains two orbits of $\langle M,\lambda_t\rangle $: one orbit with 2 elements and the other with 4 elements; hence every orbit of $\langle M,\lambda_t\rangle $ has either 2 or 4 elements,  by Proposition~\ref{4.3}. 

Let $u \leq w$, $u \in W^H$.  We need to show that
\begin{equation}
 R_{u,w}^{H,x} (q)= \left\{ \begin{array}{ll}
R_{M(u),M(w)}^{H,x}(q), & \mbox{if $M(u)  \lhd u$,} \\
(q-1)R_{u,M(w)}^{H,x}(q)+qR_{M(u),M(w)}^{H,x}(q), & \mbox{if $M(u) \rhd u$
and $M(u) \in W^{H}$,} \\
(q-1-x)R_{u,M(w)}^{H,x}(q), & \mbox{if $M(u) \rhd u$ and  $M(u) \notin W^{H}$.} 
\end{array} \right. 
\end{equation}
This is trivial if also $ |\langle M, \lambda_t \rangle (u)|=2 $ so
assume  $ |\langle M, \lambda_t \rangle (u)|=4 $.  Since $M$ agrees with $\lambda_t$ on $st$ and $tst$,  necessarily $ \langle M, \lambda_t \rangle (u)= \{y_J, sy_J, ty_J, tsy_J\} $, for a certain $y_J \in W_J$ with $s \notin D_L(y_J)$. Note that  $M(y_J)=sy_J$, $M(ty_J)=tsy_J$, and $t \not \leq y_J$ since $t\notin J$. 

Assume that  $ \langle M, \lambda_t \rangle (u)= \{y_J, sy_J, ty_J, tsy_J\}$ is contained in $W^H$ and note that, by the Word Property (Theorem~\ref{tits}),  also $sty_J$ 
is in $W^H$ since $t \not \leq y_J$ and no braid move could involve the leftmost  letters $s,t$ (so we cannot obtain a reduced expression ending with a letter in $H$). We distinguish the following 4 cases according to which of the elements in the orbit is $u$.
\begin{enumerate}

\item[(a)] $u= tsy_J$. Then
$$R^{H,x}_{u,w}= R^{H,x}_{ts  y_J,tst  w_J}= R^{H,x}_{sy_J,st  w_J}= R^{H,x}_{ty_J,st w_J } = R^{H,x}_{M(u),M(w)}, $$
where the third equality follows from Lemma~\ref{calcoli}.

\item[(b)] $u = s y_J$. Then
$$R^{H,x}_{u,w} = R^{H,x}_{sy_J,tst w_J} =(q-1) R^{H,x}_{sy_J,stw_J} + q  R^{H,x}_{tsy_J,stw_J} =(q-1) R^{H,x}_{sy_J,stw_J}=(q-1) R^{H,x}_{y_J,tw_J}$$
(where the third equality holds since $s \not \leq w_J$ implies  $tsy_J  \not \leq stw_J$) and
$$ R^{H,x}_{M(u),M(w)}=  R^{H,x}_{y_J, stw_J}= (q-1) R^{H,x}_{y_J, t w_J} + q    R^{H,x}_{sy_J, tw_J}\\
=(q-1) R^{H,x}_{y_J,tw_J}$$
(where the third equality holds since $s \not \leq w_J$  implies  $sy_J  \not \leq tw_J$).

\item[(c)] $u = t y_J$. Then
\begin{eqnarray*}
R^{H,x}_{u,w} &=& R^{H,x}_{ty_J,tst w_J} = R^{H,x}_{y_J,st w_J} =  (q-1) R^{H,x}_{y_J,tw_J} + q  R^{H,x}_{sy_J,tw_J} =  (q-1) R^{H,x}_{y_J,tw_J}\\ & =&  (q-1) \Big( (q-1) R^{H,x}_{y_J, w_J} + q  R^{H,x}_{ty_J, w_J} \Big ) = (q-1)^2 R^{H,x}_{y_J, w_J} , 
\end{eqnarray*}
where the fourth equality holds since $s \not \leq w_J$  implies  $sy_J  \not \leq tw_J$, and the last equality follows since $t \not \leq w_J$.
On the other hand 
\begin{eqnarray*}
(q-1) R^{H,x}_{u,M(w)} + q R^{H,x}_{M(u),M(w)}&=& (q-1) R^{H,x}_{ty_J,stw_J} + q R^{H,x}_{tsy_J,stw_J}\\
&=& (q-1) R^{H,x}_{ty_J,stw_J} \\
&=&(q-1)\Big( (q-1) R^{H,x}_{t y_J, tw_J} + q   R^{H,x}_{s t y_J, tw_J} \Big ) \\
&=&   (q-1)^2  R^{H,x}_{ty_J, tw_J}  \\
&=&   (q-1)^2  R^{H,x}_{y_J, w_J},
\end{eqnarray*}
where the second and the fourth equalities  hold since $s \not \leq w_J$  implies both $tsy_J  \not \leq stw_J$ and $sty_J  \not \leq tw_J$.

\item[(d)] $u = y_J$. Then
\begin{eqnarray*}
R^{H,x}_{u,w} &=& R^{H,x}_{y_J,tst w_J} =(q-1) R^{H,x}_{y_J,st  w_J} + q  R^{H,x}_{ty_J,st w_J} =(q-1) R^{H,x}_{y_J,st w_J} + q  R^{H,x}_{sy_J,st w_J} \\
&= & (q-1) R^{H,x}_{u,M(w)} + q R^{H,x}_{M(u),M(w)},
\end{eqnarray*}
where the third equality follows from Lemma~\ref{calcoli}.

\end{enumerate}

Assume now that  $ \langle M, \lambda_t \rangle (u)= \{y_J, sy_J, ty_J, tsy_J\}$ is not contained in $W^H$. The case  
 $\{y_J, sy_J, ty_J, tsy_J\} \cap W^H =\{y_J, sy_J, ty_J\} $ is impossible by Lemma~\ref{coatomi-al-piu}. 

Also the case  
 $\{y_J, sy_J, ty_J, tsy_J\} \cap W^H =\{y_J, sy_J\} $ is impossible. Indeed, in this case $sy_J$ would be the unique coatom of $ts y_J$ in $W^H$, i.e. $tsy_J= s y_J h$ for a certain $h \in H$. Since $t \not \leq s y_J$, $h=t$: being $m(s,t)>2$, this is in contradiction with the Word Property (Theorem~\ref{tits}).  

If  $\{y_J, sy_J, ty_J, tsy_J\} \cap W^H =\{y_J\} $, then clearly $u= y_J$ and 
$$ R^{H,x}_{u,w}=R^{H,x}_{y_J,tst w_J}= (q-1-x) R^{H,x}_{y_J, st w_J}=(q-1-x) R^{H,x}_{u,M(w)},$$
and the assertion follows.

It remains to treat the case  $\{y_J, sy_J, ty_J, tsy_J\} \cap W^H =\{y_J, ty_J\} $, in which we have two possibilities for $u$.

\begin{enumerate}

\item[(1)] $u = t y_J$. Then
\begin{eqnarray*}
R^{H,x}_{u,w} &=& R^{H,x}_{ty_J,tst w_J} = R^{H,x}_{y_J,stw_J} =  (q-1-x) R^{H,x}_{y_J,tw_J} \\
&=& (q-1-x) \Big( (q-1) R^{H,x}_{y_J, w_J} + q  R^{H,x}_{ty_J, w_J} \Big )  = (q-1-x)  (q-1) R^{H,x}_{y_J, w_J}, 
\end{eqnarray*}
where the last equality follows since $t \not \leq w_J$.
On the other hand 
\begin{eqnarray*}
(q-1-x) R^{H,x}_{u,M(w)} &=& (q-1-x) R^{H,x}_{ty_J,st w_J} =(q-1-x)\Big( (q-1) R^{H,x}_{t y_J, tw_J} + q   R^{H,x}_{s t y_J, tw_J} \Big )  \\
&=&(q-1-x)(q-1) R^{H,x}_{ ty_J, tw_J}=(q-1-x)(q-1) R^{H,x}_{ y_J, w_J},
\end{eqnarray*}
where 
\begin{itemize}
\item  the second equality follows by the fact that $ty_J \in W^H$ implies $sty_J \in W^H$ since all reduced expressions of  $sty_J$ start with $s$ by the Word Property (Theorem~\ref{tits}),
\item the third equality holds since $s \not \leq w_J$.
\end{itemize}
\item[(2)] $u = y_J$. Then
\begin{eqnarray*}
R^{H,x}_{u,w} &=& R^{H,x}_{y_J,tstw_J} =(q-1) R^{H,x}_{y_J,stw_J} + q  R^{H,x}_{ty_J,stw_J} \\
&=&(q-1) (q-1-x) R^{H,x}_{y_J,tw_J} + q \Big ( (q-1) R^{H,x}_{ty_J, tw_J} + q R^{H,x}_{sty_J, tw_J}\Big ) \\
&=&(q-1) (q-1-x) \Big ( (q-1) R^{H,x}_{y_J,w_J} + qR^{H,x}_{t y_J,w_J}  \Big) + q  (q-1) R^{H,x}_{y_J,   w_J}  \\
&=&(q-1) [(q-1-x)  (q-1)  + q] R^{H,x}_{y_J, w_J},  
\end{eqnarray*}
where the third equality follows (as before) by the fact that $ty_J \in W^H$ implies $sty_J \in W^H$, the  fourth and the fifth equalities follow since $s,t \not \leq w_J$. 
On the other hand, we have 
\begin{eqnarray*}
(q-1-x) R^{H,x}_{u,M(w)} &=& (q-1-x) R^{H,x}_{y_J,st w_J} =(q-1-x)^2   R^{H,x}_{ y_J, tw_J}   \\
&=&(q-1-x)^2 \Big(  (q-1) R^{H,x}_{ y_J, w_J} + q R^{H,x}_{ t y_J, w_J} \Big) \\
&=& (q-1-x)^2   (q-1) R^{H,x}_{ y_J,  w_J},
\end{eqnarray*}
where  the last equality follows since $t \not \leq w_J$.
The two expressions coincide since $(q-1-x)(q-1)+q=(q-1-x)^2$ is the equation defining $x$ (changing $x$ with $q-1-x$). 
\end{enumerate}

{\bf Case 2.}
Notice that, as in the previous  case, also now we have $w=tst \cdot \, w_{J}$ and $M(w)= st \cdot \, w_{J}$. Then  the assertion is proved by  an  argument which is very similar to the argument for the proof in Case 1. However there are some differences and we prefer to write explicitly all subcases where the differences occur.

Since $t\in D_L(w)$, $\lambda_t$ is a special matching of $w$. $M$ and $\lambda_t$ commute but they agree on $w$ (i.e., $|\langle M, \lambda_t \rangle (w)| = 2$). 
 The lower dihedral interval $[e, w_0(s,t)]$ contains 3 orbits of $\langle M,\lambda_t\rangle $: 2 orbits with 2 elements and the other with 4 elements; hence every orbit of $\langle M,\lambda_t\rangle $ must have either 2 or 4 elements,  by Proposition~\ref{4.3}. 

Let $u \leq w$, $u \in W^H$.  We need to show that
\begin{equation}
 R_{u,w}^{H,x} (q)= \left\{ \begin{array}{ll}
R_{M(u),M(w)}^{H,x}(q), & \mbox{if $M(u)  \lhd u$,} \\
(q-1)R_{u,M(w)}^{H,x}(q)+qR_{M(u),M(w)}^{H,x}(q), & \mbox{if $M(u) \rhd u$
and $M(u) \in W^{H}$,} \\
(q-1-x)R_{u,M(w)}^{H,x}(q), & \mbox{if $M(u) \rhd u$ and  $M(u) \notin W^{H}$.} 
\end{array} \right. 
\end{equation}
This is trivial if also $ |\langle M, \lambda_t \rangle (u)|=2 $ so
assume  $ |\langle M, \lambda_t \rangle (u)|=4 $.  Since $M$ agrees with $\lambda_t$ on $st$, $tst$, $sts$, and $tsts$, necessarily $ \langle M, \lambda_t \rangle (u)= \{y_J, sy_J, ty_J, tsy_J\} $, for a certain $y_J \in W_J$ with $s \notin D_L(y_J)$. Note that  $M(y_J)=sy_J$, $M(ty_J)=tsy_J$, and $t \not \leq y_J$ since $t\notin J$. 

Assume that  $ \langle M, \lambda_t \rangle (u)= \{y_J, sy_J, ty_J, tsy_J\}$ is contained in $W^H$ and note that, by the Word Property (Theorem~\ref{tits}),  also $sty_J$ and $stsy_J$ are in $W^H$ since $t \not \leq y_J$ and no braid move could involve the leftmost  letters $s,t$ (so we cannot obtain a reduced expression ending with a letter in $H$). We distinguish the following 4 cases according to which element in the orbit is $u$.
\begin{enumerate}

\item[(a)] $u= tsy_J$. Then
$$R^{H,x}_{u,w}= R^{H,x}_{ts  y_J,tst  w_J}= R^{H,x}_{sy_J,st  w_J}= R^{H,x}_{ty_J,st w_J } = R^{H,x}_{M(u),M(w)}, $$
where the third equality follows from Lemma~\ref{calcoli}.

\item[(b)] $u = s y_J$. Then
\begin{eqnarray*}
R^{H,x}_{u,w} &=& R^{H,x}_{sy_J,tst w_J} =(q-1) R^{H,x}_{sy_J,stw_J} + q  R^{H,x}_{tsy_J,stw_J} \\
&=& (q-1) R^{H,x}_{y_J, tw_J} + q \Big(   (q-1)  R^{H,x}_{tsy_J,tw_J} + q  R^{H,x}_{stsy_J,tw_J} \Big) \\
&= &  (q-1) R^{H,x}_{y_J, tw_J} + q   (q-1)  R^{H,x}_{sy_J,  w_J} = (q-1) R^{H,x}_{y_J, t w_J} + q    R^{H,x}_{sy_J, tw_J}\\
&= & R^{H,x}_{y_J, stw_J}=   R^{H,x}_{M(u),M(w)}, 
\end{eqnarray*}
where the fourth equality holds since $t \not \leq w_J$ and the fifth holds since 

$$  R^{H,x}_{sy_J, tw_J}  =(q-1) R^{H,x}_{sy_J, w_J} + q  R^{H,x}_{tsy_J,w_J} = (q-1) R^{H,x}_{sy_J, w_J}.$$

\item[(c)] $u = t y_J$. Then
\begin{eqnarray*}
R^{H,x}_{u,w} &=& R^{H,x}_{ty_J,tst w_J} = R^{H,x}_{y_J,st w_J} =  (q-1) R^{H,x}_{y_J,tw_J} + q  R^{H,x}_{sy_J,tw_J} \\
&=& (q-1) \Big( (q-1) R^{H,x}_{y_J, w_J} + q  R^{H,x}_{ty_J, w_J} \Big ) + q \Big(   (q-1)  R^{H,x}_{sy_J,w_J} + q  R^{H,x}_{tsy_J,w_J} \Big) \\
&=& (q-1)^2  R^{H,x}_{y_J, w_J}   + q  (q-1)  R^{H,x}_{sy_J,w_J} , 
\end{eqnarray*}
where the fifth equality follows since $t \not \leq w_J$.
On the other hand 
\begin{eqnarray*}
(q-1) R^{H,x}_{u,M(w)} + q R^{H,x}_{M(u),M(w)}&=& (q-1) R^{H,x}_{ty_J,stw_J} + q R^{H,x}_{tsy_J,stw_J}\\
&=&(q-1)\Big( (q-1) R^{H,x}_{t y_J, tw_J} + q   R^{H,x}_{s t y_J, tw_J} \Big ) + \\
&  &  q \Big( (q-1) R^{H,x}_{t s y_J, tw_J} + q   R^{H,x}_{s ts y_J, tw_J} \Big ) \\
&=&   (q-1)^2  R^{H,x}_{ty_J, tw_J}  + q    (q-1)  R^{H,x}_{tsy_J, tw_J}\\
&=&   (q-1)^2  R^{H,x}_{y_J, w_J}  + q    (q-1)  R^{H,x}_{sy_J, w_J},
\end{eqnarray*}
where the third equality follows since $t \not \leq w_J$.

\item[(d)] $u = y_J$. Then
\begin{eqnarray*}
R^{H,x}_{u,w} &=& R^{H,x}_{y_J,tst w_J} =(q-1) R^{H,x}_{y_J,st  w_J} + q  R^{H,x}_{ty_J,st w_J} \\
&=&(q-1) R^{H,x}_{y_J,st w_J} + q  R^{H,x}_{sy_J,st w_J} = (q-1) R^{H,x}_{u,M(w)} + q R^{H,x}_{M(u),M(w)},
\end{eqnarray*}
where the third equality follows from Lemma~\ref{calcoli}.

\end{enumerate}

Assume now that  $ \langle M, \lambda_t \rangle (u)= \{y_J, sy_J, ty_J, tsy_J\}$ is not contained in $W^H$. The case  
 $\{y_J, sy_J, ty_J, tsy_J\} \cap W^H =\{y_J, sy_J, ty_J\} $ is impossible by Lemma~\ref{coatomi-al-piu}. 

Also the case  
 $\{y_J, sy_J, ty_J, tsy_J\} \cap W^H =\{y_J, sy_J\} $ is impossible. Indeed, in this case $sy_J$ would be the unique coatom of $ts y_J$ in $W^H$, i.e. $tsy_J= s y_J h$ for a certain $h \in H$. Since $t \not \leq s y_J$, $h=t$: being $m(s,t)>2$, this is in contradiction with the Word Property (Theorem~\ref{tits}).  

If  $\{y_J, sy_J, ty_J, tsy_J\} \cap W^H =\{y_J\} $, then clearly $u= y_J$ and 
$$ R^{H,x}_{u,w}=R^{H,x}_{y_J,tst w_J}= (q-1-x) R^{H,x}_{y_J, st w_J}=(q-1-x) R^{H,x}_{u,M(w)},$$
and the assertion follows.

The subcase  $\{y_J, sy_J, ty_J, tsy_J\} \cap W^H =\{y_J, ty_J\} $ can be treated as the corresponding subcase in Case 1 and therefore we omit the proof.

{\bf Case 3.} Notice that, in this case, $s \not\leq \, ^{\{s,t\}}(^Jw)$ (as otherwise $tsts$ would be $\leq w$ by the Subword Property) and  $M(w)= st \cdot \,  ^{\{s,t\}}(^Jw)$. Then this last case is completely analogous to Case 1 and we omit the proof.
\end{proof}

\begin{rem}
Note that, in the last subcase of Case 1 and then also in the last subcase of Case 2 and Case 3 (whose proofs  are omitted)  of the proof of Theorem~\ref{doppiamente-allacciati}, we use the fact that $x$ satisfies $x^2= q + (q-1)x$. 
\end{rem}

\subsection{Dihedral Coxeter groups}
We now prove that the H-special matchings of $w$ are calculating also in the case $[e,w]$ is a dihedral interval (so, in particular, in the case $W$ is a dihedral group, i.e. a Coxeter group of rank 2).
\begin{pro}
\label{prodiedrali}
Given a Coxeter system $(W,S)$ and a subset  $H \subset S$, let $w\in W^H$ be such that $[e,w]$ is a dihedral interval and $[e,w]\cap W^H$ is a chain (i.e. a totally ordered set). Then
$$ R^{H,x}_{u,w} (q)= 
(q-1) (q-1-x) ^{\ell(w) - \ell(u)-1}$$
 for all $u < w$, $u\in W^H$.
\end{pro}
\begin{proof}
We proceed by induction on $\ell(w)$, the case $\ell(w)=1$ being clear.

 Fix a    left multiplication matching $\lambda$ of $w$. Then
$$ R^{H,x}_{u,w} = \left\{ \begin{array}{ll}
R^{H,x}_{\lambda(u),\lambda(w)}, & \mbox{if $u \rhd \lambda(u) $,} \\
(q-1)  R^{H,x}_{u,\lambda(w)}  + q  R^{H,x}_{\lambda(u),\lambda(w)}, & \mbox{if $u \lhd \lambda(u)  \in W^H$,} \\
(q-1-x)  R^{H,x}_{u,\lambda(w)} , & \mbox{if $u \lhd \lambda(u)  \notin W^H$ ,}
\end{array} \right.$$ 
In all three cases, the assertion follows easily by induction.
\end{proof}

\begin{thm}
\label{diedrali}

Given a Coxeter system $(W,S)$ and a subset  $H \subset S$, let $w\in W^H$ be such that $[e,w]$ is a dihedral interval.
 Then all $H$-special matchings of $w$ calculate the $R^{H,x}$-polynomials.
\end{thm}
\begin{proof}
The intersection $[e,w]\cap W^H$ is either  trivially equal to $[e,w]$ or is a chain. In the first case, the parabolic $R$-polynomials coincide with the ordinary $R$-polynomials and the result is known (and easy to proof). 

Assume we are in the second case and proceed by induction on $\ell(w)$. Let $M$ be an $H$-special matching of $w$ and $u \in W^H$, $u \leq w$. We need to show that
\begin{equation*}
 R_{u,w}^{H,x} (q)= \left\{ \begin{array}{ll}
R_{M(u),M(w)}^{H,x}(q), & \mbox{if $M(u)  \lhd u$,} \\
(q-1)R_{u,M(w)}^{H,x}(q)+qR_{M(u),M(w)}^{H,x}(q), & \mbox{if $M(u) \rhd u$
and $M(u) \in W^{H}$,} \\
(q-1-x)R_{u,M(w)}^{H,x}(q), & \mbox{if $M(u) \rhd u$ and  $M(u) \notin W^{H}$.} 
\end{array} \right. 
\end{equation*} 
Note that
\begin{itemize}
\item  if $ u \rhd M(u)$, then $M(u) \in W^H$ since $M$ is $H$-special,  
\item if $M(u) \rhd u$
and $M(u) \in W^{H}$, then $ M(u) \leq M(w)$ unless  $M(w)=u$.
\end{itemize}  
Then the assertion follows by induction using Proposition~\ref{prodiedrali}.
\end{proof}

\section{Combinatorial Invariance}

In this brief final section, we show the consequences of Theorems~\ref{doppiamente-allacciati} and \ref{diedrali} on the problem of the Combinatorial Invariance of the parabolic Kazhdan--Lusztig and $R$-polynomials. 
\begin{cor}
\label{congettura}
Let $(W_1,S_1)$ and $(W_2,S_2)$ be two doubly laced or dihedral Coxeter systems, with identity elements $e_1$ and $e_2$, and let $H_1 \subseteq S_1$ and 
$H_2\subseteq S_2$. Let $v_1 \in W_1^{H_1}$ and $v_2 \in W_2^{H_2}$ be such that there exists 
a poset-isomorphism  $\psi$ from $[e_1,v_1]$ to $[e_2,v_2]$ which restricts to a 
poset-isomorphism from $[e_1,v_1]^{H_1}$ to $[e_2,v_2]^{H_2}$. Then, for all $u,w \in [e_1,v_1]^{H_1}$, we have 
$$P_{u,w}^{H_1,x} (q)=P_{\psi(u),\psi(w)}^{H_2,x} (q)$$  
 (equivalently, $R_{u,w}^{H_1,x} (q)=R_{\psi(u),\psi(w)}^{H_2,x} (q)$).
\end{cor}
\begin{proof}
By Theorem \ref{7.2}, (4), the knowledge of all parabolic Kazhdan--Lusztig polynomials indexed by elements in a fixed parabolic interval is equivalent to  the knowledge of all parabolic $R$-polynomials indexed by elements in that interval. Consequently, the statement  for the parabolic Kazhdan--Lusztig polynomials is equivalent to the statement for the parabolic $R$-polynomials.
We prove the latter.

By Theorems~\ref{doppiamente-allacciati} and \ref{diedrali}, the $H$-special matchings compute the $R$-polynomials of both $W_1$ and $W_2$. Since $[e_1,v_1]$ and $[e_2,v_2]$ are isomorphic, they have the same special matchings, and, since such isomorphism $\psi$ restricts to an isomorphism from $[e_1,v_1]^{H_1}$ to $[e_2,v_2]^{H_2}$, a special matching is $H_1$-special if and only if the corresponding matching is $H_2$-special. The assertion follows from these considerations.
\end{proof}

As an immediate result of Corollary~\ref{congettura}, we obtain the following theorem which establishes Conjecture~\ref{comb-inv-con-parab} for lower intervals within the class of doubly laced and dihedral Coxeter groups.
\begin{thm}
\label{cong}
If $(W_1,S_1)$ and $(W_2,S_2)$ are either doubly laced or dihedral Coxeter systems, then Conjecture~\ref{comb-inv-con-parab} holds when $u_1$ and $u_2$ are the identity elements.
\end{thm}
\begin{proof}
Straightforward by Corollary~\ref{congettura}.
\end{proof}

{\bf Acknowledgements:}  I am grateful  to  Pietro Mongelli for sharing  information with me about his counterexample.  I would also like to thank the anonymous referee for the valuable comments, which provided insights that helped improve the paper.

\end{document}